\documentclass[12pt]{article}
\usepackage{amsmath,amssymb,amsthm,amsfonts,latexsym,amscd}
\newtheorem{theorem}{Theorem}[section]
\newtheorem{corollary}[theorem]{Corollary}
\newtheorem{lemma}[theorem]{Lemma}

\def\PG{\mbox{\rm PG}}

\begin{document}
\title{A characterization of the planes meeting a hyperbolic quadric of $\PG(3,q)$ in a conic}
\author{Bikramaditya Sahu}
\maketitle

\begin{abstract}
In this article, a combinatorial characterization of the family of planes of $\PG(3,q)$ which meet a hyperbolic quadric in an irreducible conic, using their intersection properties with the points and lines of $\PG(3,q)$, is given.\\


\end{abstract}

\section{Introduction}
Let $\PG(n,q)$ denote the $n$-dimensional Desarguesian projective space defined over a finite field of order $q$, where $q$ is a prime power. The quadrics of $\PG(n,q)$ are very interesting objects with many combinatorial properties. One of the important properties of quadrics in $\PG(n,q)$ is that subspaces can only meet a quadric in certain ways. Therefore, we may form families of subspaces that all meet a particular quadric in the same way and then we can give a characterization of that family. 

A characterization of the family of planes meeting a non-degenerate quadric in $\PG(4,q)$ is given in \cite{But}. In a series of two recent papers \cite{BHJ,BHJS}, characterizations of elliptic hyperplanes and hyperbolic hyperplanes in $\PG(4,q)$ are given. A characterization of the family of external lines to a hyperbolic quadric in $\PG(3,q)$ was given in \cite{DGDO} for all $q$ (also see \cite{IZZ} for a different characterization in terms of a point-subset of the Klein quadric in $\PG(5,q)$). In a recent paper \cite{PS}, a characterization of secant lines to a hyperbolic quadric is given for all odd $q$, $q\geq 7$.  In this paper, we give a characterization of the family of planes meeting a hyperbolic quadric in $\PG(3,q)$ in an irreducible conic.

Let $\mathcal{Q}$ be a hyperbolic quadric in $\PG(3,q)$, that is, a non-degenerate quadric of Witt index two. One can refer to \cite{Hir} for the basic properties of the points, lines and planes of $\PG(3,q)$ with respect to $\mathcal{Q}$. The quadric $\mathcal{Q}$ consists of $(q+1)^2$ points and $2(q+1)$ lines. Every plane of $\PG(3,q)$ meets $\mathcal{Q}$ in an irreducible conic or in two intersecting generators. In the first case we call the plane as a {\it secant} plane, otherwise, a {\it tangent} plane. There are $q^3-q$ secant planes and $(q+1)^2$ tangent planes. Each point of $\PG(3,q)\setminus \mathcal{Q}$ lies on $q^2$ secant planes and each point of $\mathcal{Q}$ lies on $q^2-q$ secant planes. Each line of $\PG(3,q)$ lies on $0$, $q-1$, $q$ or $q+1$  secant planes. 

In this paper, we prove the following characterization theorem.

\begin{theorem} \label{main}
Let $\Sigma$ be a non empty family of planes of $\PG(3,q)$, for which the following properties are satisfied:
\begin{enumerate}
\item [(P1)] Every point of $\PG(3,q)$ lies on $q^2-q$ or $q^2$ planes of $\Sigma$.
\item [(P2)] Every line of $\PG(3,q)$ lies on $0$, $q-1$, $q$ or $q+1$ planes of $\Sigma$.
\end{enumerate}
Then $\Sigma$ is the set of all planes of $\PG(3,q)$ meeting a hyperbolic quadric in an irreducible conic.
\end{theorem}

\section{Preliminaries}

Let $\Sigma$ be a nonempty set of planes of $\PG(3,q)$ for which the properties (P1) and (P2) stated in Theorem \ref{main} hold. A point of $\PG(3,q)$ is said to be {\it black} or {\it white} according as it contained in $q^2-q$ or $q^2$ planes of $\Sigma$. Let $b$ and $w$, respectively, denote the number of black points and the number of white points in $\PG(3,q)$. We have 
\begin{equation}\label{eq-1}
b+w=q^3+q^2+q+1.
\end{equation}

Counting in two ways the point-plane incident pairs,
$$\{(x,\pi)\ | \ x \ \mbox{is a point in $\PG(3,q)$}, \ \pi \in \Sigma \ \mbox{and}\ x\in \pi \},$$
we get  
\begin{equation}\label{eq-2}
b(q^2-q)+w q^2=|\Sigma|(q^2+q+1).
\end{equation}

Again, counting in two ways the incident triples,
$$\{(x,\pi,\sigma)\ | \ x \ \mbox{is a point in $\PG(3,q)$}, \ \pi,\sigma \in \Sigma, \pi\neq \sigma \ \mbox{and}\ x\in \pi\cap\sigma \},$$
we get  
\begin{equation}\label{eq-3}
b(q^2-q)(q^2-q-1)+wq^2(q^2-1)=|\Sigma|(|\Sigma|-1)(q+1).
\end{equation}

If $q$ is odd, then dividing both sides of Equation (\ref{eq-3}) by $2$ to get the following equality. 
\begin{equation}\label{eq-3'}
bq(q-1)/2(q^2-q-1)+wq^2(q-1)(q+1)/2=|\Sigma|(|\Sigma|-1)(q+1)/2.
\end{equation}

\begin{lemma}\label{lem-b}
\begin{enumerate}
\item[(i)] If $q$ is even, then $q+1$ divides $b$.
\item[(ii)] If $q$ is odd, then $(q+1)/2$ divides $b$.
\end{enumerate}
\end{lemma}

\begin{proof}
Since $q^2-q-1=(q+11)(q-2)+1$, we have $q+1$ is co-prime to $q^2-q-1$. Also $q+1$ is co-prime to $q$ being consecutive integers. If $q$ is even, then $q+1$ and $q-1$ are both odd and their difference is two and so they are co-prime to each other. If $q$ is odd, then $(q+1)/2$ and $(q-1)/2$ are co-prime being consecutive integers (also $(q+1)/2$ and $q$ are co-prime). Now (i) and (ii) follows from Equations (\ref{eq-3}) and (\ref{eq-3'}), respectively.
\end{proof}

\begin{lemma}\label{lem-black-fixed}
Every plane of $\Sigma$ contains a fixed number of black points.
\end{lemma}

\begin{proof}
Let $\pi$ be a plane in $\Sigma$. Let $b_\pi$ and $w_\pi$, respectively, denote the number of black and white points in $\pi$. Then $w_\pi=q^2+q+1-b_\pi$.

By counting the incident point-plane pairs of the following set,
$$\{(x,\sigma)\ | \ x \ \mbox{is a point in $\PG(3,q)$}, \ \sigma \in \Sigma \ , \sigma\neq \pi \ \mbox{and} \ x\in \pi\cap \sigma \},$$ we get\\ 
$$b_\pi(q^2-q-1)+w_\pi(q^2-1)=(|\Sigma|-1)(q+1).$$
Since $q+1$ is co-prime to $q^2-q-1$ (see the proof of Lemma \ref{lem-b}), it follows from the above equation that $q+1$ divides $b_\pi$. Let $b_\pi=(q+1)r_\pi$ for some $0\leq r_\pi\leq q$. 

Since $w_\pi=q^2+q+1-b_\pi=q^2+q+1-(q+1)r_\pi$, by a simple calculation, we get  
$$q^3-qr_\pi=|\Sigma|.$$
Since $|\Sigma|$ is a fixed number, $r_\pi:=r$ is fixed. Hence $b_\pi=(q+1)r$ is a fixed number. 
\end{proof}

From the proof of Lemma \ref{lem-black-fixed},  
\begin{equation}\label{eq-5}
q^3-qr=|\Sigma|,
\end{equation}
where the number of black points in a plane in $\Sigma$ is $(q+1)r$ for some $0\leq r\leq q$. 

\begin{lemma}\label{black-tangent}
Any plane of $\PG(3,q)$ not in $\Sigma$ contains $q+(q+1)r$ black points. 
\end{lemma}

\begin{proof}
We prove the lemma with a similar argument as that of Lemma \ref{lem-black-fixed}. Let $\pi$ be a plane of $\PG(3,q)$ not in $\Sigma$. Let $b_\pi$ and $w_\pi$ denote the number of black and white points, respectively, in $\pi$. Then $w_\pi=q^2+q+1-b_\pi$.

By counting the incident point-plane pairs of the following set,
$$\{(x,\sigma)\ | \ x \ \mbox{is a point in $\PG(3,q)$}, \ \sigma \notin \Sigma \ , \sigma\neq \pi \ \mbox{and} \ x\in \pi\cap \sigma \},$$ we get\\ 
$$b_\pi(2q+1-1)+w_\pi (q+1-1)=(q^3+q^2+q+1-|\Sigma|-1)(q+1).$$
From Equation (\ref{eq-5},) $|\Sigma|=q^3r-qr$. The above equation simplifies to 
$$2b_\pi+w_\pi=(q+1+r)(q+1).$$
Since $w_\pi=q^2+q+1-b_\pi$, we get $b_\pi=q+(q+1)r$.
\end{proof}

\begin{corollary}\label{r-bound}
$1\leq r\leq q-1$.
\end{corollary}

\begin{proof}
Note that $0\leq r\leq q$. Suppose $r=0$. By Equation (\ref{eq-5}), we get $|\Sigma|=q^3$. From Equation (\ref{eq-2}), it follows that 
$$b(q-1)+qw=q^2(q^2+q+1).$$ Putting $w=q^3+q^2+q+1-b$ (Equation (\ref{eq-1})) in the above equation to get $b=q$. But then from Equation (\ref{eq-3}), $q+1$ divides $q^2(q-1)$. Since $q$ is co-prime to $q+1$, we see that $q+1$ divides $q-1$ which is a contradiction.

Suppose that $r=q$. Then by Lemma \ref{black-tangent}, any plane of $\PG(3,q)$ not in $\Sigma$ contains $q^2+2q$ black points, which is a contradiction. Hence $1\leq r\leq q-1$.
\end{proof}

\begin{lemma}\label{size-Sigma}
We have $r=1$, $b=(q+1)^2$ and $|\Sigma|=q^3-q$. In particular, every plane in $\Sigma$ contains $(q+1)$ black points.
\end{lemma}

\begin{proof}
Note that, by Corollary \ref{r-bound}, $1\leq r\leq q-1$. If $q=2$, then $r=1$. Assume first that $q$ even, $q\geq 4$.
 
By Lemma \ref{lem-b}(i) and Equation (\ref{eq-3}), it follows that $q(q-1)$ divides $|\Sigma|(|\Sigma|-1)$. Since $|\Sigma|=q^3-qr$ (Equation (\ref{eq-5})), we have $q(q-1)$ divides $(q^3-qr)(q^3-qr-1)$, i.e, $q-1$ divides $(q^2-r)$ or $(q^3-qr-1)$. 

Note that $q^2-r=(q^2-1)-(r-1)$ and $q^3-qr-1=(q^3-q)+(q-1)-qr$. If $q-1$ divides $(q^2-r)$, then $r=1$ or $r=q$. If $q-1$ divides $q^3-qr-1$, then $r=0$ or $r=q-1$.  Since $1\leq r\leq q-1$, we have $r=1$ or $q-1$. 

Suppose $r=q-1$. putting $|\Sigma|=q^3-qr$ in Equation (\ref{eq-2}) we get that
$$b(q-1)+qw=(q^2-q+1)(q^2+q+1).$$ Since $w=q^3+q^2+q+1-b$, solving the above equation for $b$, it follows that  
$b=q^3+q-1=q^2(q+1)-(q-1)(q+1)+(q-2)$, which is a contradiction to Lemma \ref{lem-b}(i) for $q\geq 4$. Hence $r=1$ for all $q$ even.\medskip

Now assume that $q$ is odd.\medskip 
 
By Equation (\ref{eq-3'}) and Lemma \ref{lem-b}(ii), we see that $q(q-1)/2$ divides $|\Sigma|(|\Sigma|-1)$. Thus $q(q-1)/2$ divides $(q^3-qr)(q^3-qr-1)$, i.e, $(q-1)/2$ divides $(q^2-r)$ or $(q^3-qr-1)$.

As before $q^2-r=(q^2-1)-(r-1)$ and $q^3-qr-1=(q^3-q)+(q-1)-qr$. If $(q-1)/2$ divides $(q^2-r)$, then $r=1$ or $r=(q+1)/2$. If $(q-1)/2$ divides $q^3-qr-1$, then $r=(q-1)/2$ or $r=q-1$. So, $r=1,(q-1)/2, (q+1)/2$ or $q-1$. 

Suppose that $r=q-1$. As similar to $q$ even case, here also we get $b=q^3+q-1=(q^3+q^2)-(q^2-1)+(q-2)$. By Lemma \ref{lem-b}(ii), $(q+1)/2$ divides $q-2$ i.e, $(q+1)/2$ divides $q+1-3$. This gives $q=5$. So $b=5^3+5-1=129$, $r=4$ and hence $|\Sigma|=5^3-5\cdot 4=105$. Now $w=5^3+5^2+5+1-b=27$. On the other hand, putting the values of $b=129$ and $|\Sigma|=105$ in Equation (\ref{eq-3}), we get that $w=\dfrac{55}{2}$, which is a contradiction.  

Suppose that $r=(q-1)/2$. Then $|\Sigma|=q^3-q(q-1)/2$. By a similar calculation as before using Equations (\ref{eq-1}) and (\ref{eq-2}), we get that $b=(q^3+2q-1)/2=(q^3+1)/2+(q-1)$. By Lemma \ref{lem-b}(ii), $(q+1)/2$ divides $q-1$ i.e, $(q+1)/2$ divides $q+1-2$. This gives $q=3$ and hence $r=(q-1)/2=1$.  

Suppose now that $r=(q+1)/2$. Again by using Equations (\ref{eq-1}) and (\ref{eq-2}), we get that $b=(q^3+2q^2+4q+1)/2=(q^3+1)/2+q(q+1)+q$. Again by Lemma \ref{lem-b}(ii), $(q+1)/2$ divides $q$, which is a contradiction as $q$ and $q+1$ are co-prime. Hence $r=1$ for all $q$ odd.  

Hence for all $q$, we have $r=1$ and $|\Sigma|=q^3-q$. Putting the values of $|\Sigma|$ and $w=q^3+q^2+q+1-b$ in Equation (\ref{eq-2}), it follows that $b=(q+1)^2$. Since $r=1$, every plane in $\Sigma$ contains $(q+1)$ black points.
\end{proof}

\begin{corollary}\label{tangent-black}
\begin{enumerate}
\item[(i)] There are $(q+1)^2$ planes in $\PG(3,q)$ which are not in $\Sigma$.
\item[(ii)] Every plane in $\PG(3,q)$ not in $\Sigma$ contains $2q+1$ black points.
\end{enumerate}
\end{corollary}

\begin{proof}
(i) follows from Lemma \ref{size-Sigma} and the fact that there are $(q^2+1)(q+1)$ planes in $\PG(3,q)$. Now (ii) follows from Lemma \ref{black-tangent}, since $r=1$.
\end{proof}

We call a plane of $\PG(3,q)$ {\it tangent} if it is not a plane in $\Sigma$. By Corollary \ref{tangent-black}, every tangent plane contains $2q+1$ black points.
 
\begin{lemma}\label{lem-plane-point}
Let $l$ be a line of $\PG(3,q)$. Then the number of tangent planes through $l$ is equal to the number of black points contained in $l$.
\end{lemma}

\begin{proof}
Let $t$ and $s$, respectively, denote the number of tangent planes through $l$ and the number of black points contained in $l$.  We count in two different ways the point-plane incident pairs,
$$\{(x,\pi)\ | \ x\in l, \ \pi\ \mbox{is a tangent plane}\ \mbox{and}\ x\in \pi \}.$$
Hence  
$$s(2q+1)+(q+1-s)(q+1)=((q+1)^2-t)\cdot 1+t\cdot (q+1).$$
It follows that $s=t$. This proves the lemma.
\end{proof}

\begin{corollary}\label{cor-line-black}
Every line of $\PG(3,q)$ contains $0,1,2$ or $q+1$ black points.
\end{corollary}

\begin{proof}
From Theorem \ref{main}(P2), every line is contained in $0,1,2$ or $q+1$ tangent planes. The proof now follows from Lemma \ref{lem-plane-point} 
\end{proof}

As a consequence of Lemma \ref{lem-plane-point}, we have the following.

\begin{corollary}\label{cor-plane-1}
Every black point is contained in some tangent plane.
\end{corollary}

\section{Black lines}
Let $B$ be the set of all black points in $\PG(3,q)$. We call a line {\it black} if all its $q+1$ points are contained in $B$.

\begin{lemma}\label{tangent-plane-black}
Let $\pi$ be a tangent plane. Then the set $\pi\cap B$ is a union of two (intersecting) black lines.
\end{lemma}

\begin{proof}
By Corollary \ref{tangent-black}(ii), we have $|\pi\cap B|=2q+1$. Then $\pi\cap B$ is a not an arc, since an arc in any plane has at most $q+2$ points. Thus there is a line $l$ in $\pi$ such that $|l\cap B|\geq 3$. By Corollary \ref{cor-line-black}, $l\subseteq B$.

Let $x$ be a point on $l$. Since $|\pi\cap B|=2q+1$ and $l\subseteq B$, other $q$ points (points not on $l$) of $\pi\cap B$ lie on $q$ lines through $x$, different from $l$. If there is a line $m(\neq l)$ through $x$ containing two points of $\pi\cap B\setminus \{x\}$, then  $|m\cap B|\geq 3$. By Corollary \ref{cor-line-black}, $m\subseteq B$. So, in this case $\pi\cap B=l\cup m$. On the other hand, if there are two points $y$ and $z$ (different from $x$) of $B$ which lie on two different lines (different from $l$) through $x$, then the line $m:=yz$ intersects $l$ at a point different from $y$ and $z$. In particular $|m\cap B|\geq 3$. Again by Corollary \ref{cor-line-black}, $m\subseteq B$. In any case, $\pi\cap B=l\cup m$. This proves the lemma.
\end{proof} 

\begin{lemma}\label{black-line-1}
Every black point lies on at most two black lines.
\end{lemma}

\begin{proof}
Let $x$ be a black point. If possible, suppose that there are three distinct black lines $l,m,k$ each of which contains $x$. 

Note that, by Lemma \ref{lem-plane-point}, each of the $q+1$ planes through $l$ as well as $m$ are tangent planes giving rise to $2q+2$ tangent planes through $x$ (with repetition allowed). Since the tangent plane $\langle l,m \rangle$ contains both $l$ and $m$, each of the $2q+1$ tangent planes through $x$ passes through either $l$ or $m$. 

On the other hand, again by  Lemma \ref{lem-plane-point}, there are $q+1$ tangent planes through $k$. Out of these $q+1$ planes through $k$, the planes $\langle k,l\rangle$ and $\langle k,m\rangle$ have already been counted as planes through $x$. Since $q\geq 2$, there is a tangent plane through $k$ (and hence through $x$), different from $\langle k,l\rangle$ and $\langle k,m\rangle$. This gives a contradiction to the fact that each of the tangent plane through $x$ is either passes through $l$ or $m$. Hence, there are at most two black lines through through $x$.   
\end{proof}

\begin{lemma}\label{black-line-2}
Every black point is contained in precisely two black lines.
\end{lemma}

\begin{proof}
Let $x$ be a black point and $l$ be a black line containing $x$. The existence of such a line $l$ follows from the facts that $x$ is contained in a tangent plane (Corollary \ref {cor-plane-1}) and that the set of all black points in that tangent plane is a union of two black lines (Lemma \ref{tangent-plane-black}). By Lemma \ref{lem-plane-point}, let $\pi_1,\pi_2,\ldots,\pi_{q+1}$ be the $q+1$ tangent planes through $l$. For $1\leq i\leq q+1$, by Lemma \ref{tangent-plane-black}, we have $B\cap\pi_i=l\cup l_i$ for some black line $l_i$ of $\pi_i$ different from $l$. Let $\{p_i\}=l\cap l_i$. Lemma \ref{black-line-1} implies that $p_i\neq p_j$ for $1\leq i\neq j\leq q+1$, and so $l=\{p_1,p_2,\ldots, p_{q+1}\}$. Since $x\in l$, we have $x=p_j$ for some $1\leq j\leq q+1$. Thus, $x$ is contained in precisely two black lines, namely, $l$ and $l_j$.
\end{proof}

\section{Proof of Theorem \ref{main}}

We refer to \cite{PT} for the basics on finite generalized quadrangles. Let $s$ and $t$ be positive integers. A {\it generalized quadrangle} of order $(s,t)$ is a point-line geometry $\mathcal{X}=(P,L)$ with point set $P$ and line set $L$ satisfying the following three axioms:
\begin{enumerate}
\item[(Q1)] Every line contains $s+1$ points and every point is contained in $t+1$ lines.
\item[(Q2)] Two distinct lines have at most one point in common (equivalently, two distinct points are contained in at most one line).
\item[(Q3)] For every point-line pair $(x,l)\in P\times L$ with $x\notin l$, there exists a unique line $m\in L$ containing $x$ and intersecting $l$.
\end{enumerate}

Let $\mathcal{X}=(P,L)$ be a generalized quadrangle of order $(s,t)$. Then, $|P|=(s+1)(st+1)$ and $|L|=(t+1)(st+1)$ \cite[1.2.1]{PT}. If $P$ is a subset of the point set of some projective space $PG(n,q)$, $L$ is a set of lines of $PG(n,q)$ and $P$ is the union of all lines in $L$, then $\mathcal{X}=(P,L)$ is called a {\it projective generalized quadrangle}. The points and the lines contained in a hyperbolic quadric in $\PG(3,q)$ form a projective generalized quadrangle of order $(q,1)$. Conversely, any projective generalized quadrangle of order $(q,1)$ with ambient space $\PG(3,q)$ is a hyperbolic quadric in $\PG(3,q)$, this follows from \cite[4.4.8]{PT}.\\

The following two lemmas complete the proof of Theorem \ref{main}.

\begin{lemma}\label{hyperbolic}
The points of $B$ together with the black lines form a hyperbolic quadric in $\PG(3,q)$.
\end{lemma}

\begin{proof}
We have $|B|=b=(q+1)^2$ by Lemma \ref{size-Sigma}. It is enough to show that the points of $B$ together with the black lines form a projective generalized quadrangle of order $(q,1)$.

Each black line contains $q+1$ points of $B$. By Lemma \ref{black-line-2}, each point of $B$ is contained in exactly two black lines. Thus the axiom (Q1) is satisfied with $s=q$ and $t=1$. Clearly, the axiom (Q2) is satisfied.

We verify the axiom (Q3). Let $l=\{x_1,x_2,\ldots,x_{q+1}\}$ be a black line and $x$ be a black point not contained in $l$. By Lemma \ref{black-line-2}, let $l_i$ be the second black line through $x_i$ (different from $l$) for $1\leq i\leq q+1$. If $l_i$ and $l_j$ intersect for $i\neq j$, then the tangent plane $\pi$ generated by $l_i$ and $l_j$ contains $l$ as well. This implies that $\pi\cap B$ contains the union of three distinct black lines (namely, $l,l_i,l_j$), which is not possible by Lemma \ref{tangent-plane-black}. Thus the black lines $l_1,l_2,\ldots,l_{q+1}$ are pairwise disjoint. These $q+1$ black lines contain $(q+1)^2$ black points and hence their union must be equal to $B$. In particular, $x$ is a point of $l_j$ for unique $j\in\{1,2,\ldots,q+1\}$. Then $l_j$ is the unique black line containing $x_j$ and intersecting $l$.

From the above, it follows that the points of $B$ together with the black lines form a projective generalized quadrangle of order $(q,1)$. This completes the proof.
\end{proof}

\begin{lemma}
The set of planes in $\Sigma$ are the planes meeting $B$ in a conic.
\end{lemma}

\begin{proof}
Let $\pi$ be a plane in $\Sigma$. Note that $|\pi\cap B|=q+1$ by Lemma \ref{size-Sigma}. Suppose that $l$ is a line of $\pi$ containing three points of $\pi\cap B$. Then by Corollary \ref{cor-line-black}, $l$ is contained in $\pi\cap B$.  Since $|\pi\cap B|=q+1$, we have $l=\pi\cap B$. By Lemma \ref{lem-plane-point}, all the planes through $l$ are tangent, which is a contradiction to the fact that $\pi$ is not a tangent plane. Hence $|l\cap B|\leq 2$. Since the line $l$ in $\pi$ was arbitrary, $\pi\cap B$ is an oval in $\pi$. Since $B$ is a quadric by Lemma \ref{hyperbolic}, the planes of $\Sigma$ are exactly those that meet $B$ in an irreducible conic.    
\end{proof}

\vskip .5cm

\noindent {\bf Address}:\\

\noindent {\bf Bikramaditya Sahu} (Email: sahuba@nitrkl.ac.in)\\
 Department of Mathematics, National Institute of Technology\\
 Rourkela - 769008, Odisha, India.\\
\end{document}